\documentclass[12pt,a4paper]{amsart}
\usepackage{amsfonts}
\usepackage{amsthm}
\usepackage{amsmath}
\usepackage{amssymb} %Arrows
\usepackage{amscd}
\usepackage[latin2]{inputenc}
\usepackage{t1enc}
\usepackage[mathscr]{eucal}
\usepackage{indentfirst}
\usepackage{graphicx}
\usepackage{graphics}
\usepackage{pict2e}
\usepackage{color}
\numberwithin{equation}{section}
\usepackage[margin=2.4cm]{geometry}
\usepackage{epstopdf} 
\usepackage{tikz-cd} %Diagrams
\usepackage{pdflscape} %Querformat für die großen Matrizen

\theoremstyle{plain}
\newtheorem{Th}{Theorem}[section]
\newtheorem{Lemma}[Th]{Lemma}
\newtheorem{Cor}[Th]{Corollary}
\newtheorem{Prop}[Th]{Proposition}

 \theoremstyle{definition}
\newtheorem{Def}[Th]{Definition}

\newtheorem{Rem}[Th]{Remark}
\newtheorem{?}[Th]{Problem}

\DeclareMathOperator{\expo}{exp}
\DeclareMathOperator{\Ad}{Ad}

\DeclareMathOperator{\rk}{rk}
\DeclareMathOperator{\diag}{diag}

\DeclareMathOperator{\SU}{\mathsf{SU}}

\DeclareMathOperator{\su}{\mathfrak{su}} 

\begin{document}

\title{ON THE SIGNATURE OF BIQUOTIENTS}

\author[Goertsches, Schmitt]{OLIVER GOERTSCHES AND MAXIMILIAN SCHMITT}

\address{Philipps-Universit\"at Marburg \\ Fachbereich Mathematik und Informatik \\ Hans-Meerwein-Stra\ss e \\ 35043 Marburg} 

\email{goertsch@mathematik.uni-marburg.de, schmittt@mathematik.uni-marburg.de}

 %\subjclass[2010]{Primary: 05C??. Secondary: 05C??}

 %\keywords{Eqivariant cohomology} 

\begin{abstract}We generalize Hirzebruch's computation of the signature of equal rank homogeneous spaces to a large class of biquotients.
\end{abstract}

\maketitle

\section{Introduction}

The signature of a homogeneous space $G/H$, where $H\subset G$ are compact Lie groups of equal rank, is explicitly computable from the root systems of $G$ and $H$. This was shown by Hirzebruch \cite{Hirzebruch}, as a corollary of a more general result for compact oriented manifolds on which a circle acts with finite fixed point set, see Theorem \ref{hirzebruchformula} below.

In this note we generalize Hirzebruch's computation to a large class of equal rank biquotients, i.e., quotients of a compact Lie group $G$ by the free action of a subgroup $H\subset G\times G$ with $\rk H = \rk G$ by left and right multiplication. In this way we continue the topological study of biquotients by extending methods from homogeneous spaces, which already lead to an understanding of the Euler characteristic \cite{singhof}, cohomology \cite{Esch1}, and rational homotopy \cite{rathomot} of biquotients. 

Biquotients were originally considered by Eschenburg \cite{Esch2} in the context of Riemannian geometry, but also appear naturally in other geometries, such as symplectic \cite{GKZ2} or Sasakian geometry \cite{boyer}. In all these considerations, symmetries play an essential role. We will use the fact that any Lie subgroup of $G\times G$ that commutes with $H$ naturally acts on $G//H$, yielding in particular circle actions on many such biquotients. Our main result, Theorem \ref{signature}, is applicable to any such circle action with finite fixed point set. The main difference to the homogeneous setting is the fact that because we do not have a transitive action on the space at our disposal, we need to keep track of orientations, see Definition \ref{weightremark} below. To illustrate this issue, we have included a detailed example, see Section \ref{subsec:ex}. \\

\noindent \emph{Acknowledgements.} The results of this paper are contained in the master thesis of the second named author, written at the Philipps University of Marburg under the supervision of the first named author.

\section{Actions on Homogeneous Spaces}\label{sec:homspace}
In this section we present the known results on homogeneous spaces from \cite{Hirzebruch}.\\

Consider $G$ a compact, connected Lie group and $H \subset G$ a subgroup with $\rk(H)=\rk(G)$. Fix a shared maximal torus $T \subset H \subset G$. Left multiplication with elements of the torus induces a well-defined action of $T$ on the homogeneous space $G/H$ by $t \cdot gH := (tg)H$. The fixed point set of this action is well-known and in particular finite:

\begin{Prop}\label{prop:fphomcase} The natural map $N_G(T) \to G \to G/H$ induces a bijection $(G/H)^{T} \cong N_G(T)/N_H(T) \cong \frac{W(G)}{W(H)}$.
\end{Prop}

\begin{proof}
See e.g.\ \cite[Proposition\ 2.2]{GHZ}
\end{proof}

We now want to understand the weights of the isotropy representation in the fixed points. Denote by 
\[
\pi \colon G \longrightarrow G/H
\]
the natural projection. Then:
%We can understand the isotropy representation through the following proposition:

\begin{Prop}
Let $g \in N_{G}(T)$. Then for any $t\in T$ and $v\in T_{gH}G/H$ we have
\begin{equation*}
dt_{gH}(v)=d\pi_{g}d(l_{g})_{e}\Ad_{w^{-1}(t)}(X)
\end{equation*}
 where $X\in \mathfrak g$ satisfies $d\pi_g(X_g)= v$ and
 $w^{-1}(t)=g^{-1}tg$.
\end{Prop}

\begin{proof}
For such a fixed point we define $w^{-1}(t) := g^{-1}tg\in T$. Then:
\begin{align*}
dt_{gH}(v)&=\frac{d}{ds}\Bigr|_{s=0}t g \expo(sX)H =\frac{d}{ds}\Bigr|_{s=0}g w^{-1}(t) \expo(sX)H\\
&=\frac{d}{ds}\Bigr|_{s=0}g w^{-1}(t)  \expo(sX) (w^{-1}(t))^{-1} H = d\pi_{g}d(l_{g})_{e}\Ad_{w^{-1}(t)}(X).
\end{align*}

\end{proof}

\begin{Rem}\label{rem:orientationhomspace}
Let $\Delta_H\subset \Delta_G$ be the root systems of $H$ and $G$ with respect to $T$. The former proposition tells us that the weights of the isotropy representation in each fixed point $gH$, where $g\in N_G(T)$, are the roots $\Delta_{G}\setminus\Delta_{H}$, up to sign, twisted by a representative of the fixed point, i.e. $\{\Ad_{g^{-1}}^{*}\alpha \mid \alpha \in \Delta_{G}\setminus\Delta_{H}\}$. See also \cite{GHZ}, where even more information was obtained, in form of the GKM graph of the $T$-action on $G/H$.

Let us assume that $H$ is connected. A choice of positive roots $\Delta^{+}_{G} \subset \Delta_{G}$ induces an orientation of $G/H$ as follows: the weight space decomposition of $G$ yields a decomposition
\begin{equation*}
\mathfrak{g} = \mathfrak{t} \oplus \bigoplus\limits_{\alpha \in \Delta^{+}_{G}}(\mathfrak{g}^{\mathbb{C}}_{\alpha} \oplus \mathfrak{g}^{\mathbb{C}}_{-\alpha})\cap \mathfrak{g},
\end{equation*}
hence
\[T_{eH}G/H \cong \bigoplus\limits_{\alpha \in \Delta_{G}^{+} \setminus \Delta_{H}}(\mathfrak{g}^{\mathbb{C}}_{\alpha} \oplus \mathfrak{g}^{\mathbb{C}}_{-\alpha})\cap \mathfrak{g},\]
which is the same as the decomposition of $T_{eH} G/H$ into the irreducible submodules of the isotropy representation of $T$ at $eH$. Each $\mathfrak{g}^{\mathbb{C}}_{\alpha}$ is one-dimensional and $\mathfrak{g}^{\mathbb{C}}_{-\alpha}=\overline{\mathfrak{g}^{\mathbb{C}}_{\alpha}}$.  Hence, when choosing basis vectors
\begin{align*}
&\mathfrak{g}^{\mathbb{C}}_{\alpha}=\langle X+iY \rangle_{\mathbb{C}}=\langle X,iY \rangle_{\mathbb{R}} \\
 &\mathfrak{g}^{\mathbb{C}}_{-\alpha}=\langle X-iY \rangle_{\mathbb{C}}=\langle X,-iY \rangle_{\mathbb{R}},
\end{align*}
 the choice of $\pm \alpha$ as positive corresponds to the choice of a real basis $\{X,\pm Y\}$ of \[(\mathfrak{g}^{\mathbb{C}}_{\alpha} \oplus \mathfrak{g}^{\mathbb{C}}_{-\alpha})\cap \mathfrak{g}\] and therefore gives an orientation of this two-dimensional real vector space. In total this induces an orientation of the vector space $T_{eH}G/H$, and since $G$ acts transitively on $G/H$ by left multiplication, we get an orientation of the homogeneous space $G/H$ (This will not work analogously for biquotients).
It is convenient to consider $\frac{1}{i}\alpha$ for every root $\alpha$ whenever we make use of the roots as real functionals on the Lie algebra of the maximal torus, because $\alpha$ has purely imaginary values on the Lie algebra of maximal torus as simultaneous eigenvalue of skew-symmetric endomorphisms.
\end{Rem}

This data is now sufficient to understand the signature of these spaces, defined by
\begin{Def}
Let $M$ be a compact, connected, orientable manifold of dimension $4n$. By Poincar\'e duality, multiplication in the middle cohomology defines a bilinear, symmetric, non-degenerate product
\begin{equation*}
\wedge \colon H^{2n}(M,\mathbb{R}) \times H^{2n}(M,\mathbb{R}) \longrightarrow H^{4n}(M,\mathbb{R}) \cong \mathbb{R}.
\end{equation*} 
We define  the \emph{signature} $\sigma(M)$ of $M$ to be the signature of this inner product. We set the signature of manifolds whose dimension is not divisible by four to zero.  
\end{Def}

\begin{Rem}
When $\bar{M}$ denotes $M$ with the reversed orientation, $\sigma(\bar{M})=-\sigma(M)$.
\end{Rem}

Hirzebruch computed this (oriented-homotopy) invariant using the famous Atiyah-Singer-Index Theorem \cite[p.\ 63--72]{hirzbruch}. For the special case of $S^{1}$-manifolds with finite fixed point set he obtained in \cite[Section\ 1.7.b)]{Hirzebruch}:

\begin{Th}\label{hirzebruchformula}
Take M a compact, oriented, $2n$-dimensional manifold on which $S^{1}$ acts with isolated fixed points. Denote by $V(m_{i}) \cong \mathbb{C}$ the oriented real $S^{1}$-module defined by $z \cdot v := z^{m_{i}}v$. Then, in each fixed point $p \in M^{S^{1}}$, we can decompose $T_{p}M \cong \bigoplus\limits_{i}V(m_{i})$, such that the orientations on the $V(m_{i})$ induce the given orientation on $T_{p}M$. Then these $m_{i}$ are well-defined up to an even number of sign changes and

\begin{equation*}
\sigma(M)=\sum\limits_{p \in M^{S^{1}}}(-1)^{\#\{i \mid m_{i}<0\}}.
\end{equation*}
\end{Th}

\begin{Rem}
A different choice of the $m_{i}$ does not change the parity of $\#\{i \mid m_{i}<0\}$.
\end{Rem}

If we feed in the results on the canonical torus action on equal rank homogeneous spaces, restrict our torus action to a circle which has the same fixed points as the torus, and fix sets of positive roots $\Delta_{G}^{+}$ on $G$ and $\Delta_{H} \subset \Delta_G$ on $H$ which induce an orientation on $G/H$ as described in Remark \ref{rem:orientationhomspace}, Hirzebruch's fomula yields \cite[Theorem 2.5.]{Hirzebruch}:

\begin{Th}
$\sigma(G/H)= \pm \sum\limits_{[w] \in \frac{W(G)}{W(H)}}(-1)^{\#\{\alpha \in \Delta^{+}_{G} \setminus \Delta_{H} \mid w^{-1}(\alpha)\not\in \Delta^{+}_{G}\}}$
\end{Th}

This formula is then used in numerous papers (e.g.\ \cite{App1, App2}) to compute the signature of homogeneous spaces. In the following sections we will generalize this result to a large class of biquotients.

\section{Actions on Biquotients}

In the following $G$ will always denote a compact, connected Lie group, with maximal torus $T_{\max} \subset G$. Furthermore $T$ shall denote a torus in $T_{\max} \times T_{\max}$ of dimension equal to the rank of $G$. We fix a complementary torus $T'$ in $T_{\max}\times T_{\max}$, i.e. $\mathfrak{t}\oplus \mathfrak{t'}=\mathfrak{t}_{\max} \oplus \mathfrak{t}_{\max}$. Let $H \subset G \times G$ be a closed, connected subgroup containing $T$ with $\rk G = \rk H$. We assume that $H$ (or, equivalently, $T$) acts freely on $G$ by $(h_{1},h_{2})\cdot g=h_{1}gh_{2}^{-1}$, and we denote the $H$-orbit space by $G//H$. It is called a \emph{biquotient}. We assume that $H$ commutes with a subtorus $\tilde{T} \subset T'$, so that we get a well-defined action of $\tilde{T}$ on the biquotient $G//H$ via $(t_{1},t_{2})Hg=H(t_{1}gt_{2}^{-1})$. The aim of this section is to understand the weights of the isotropy representation of this action in the fixed points.

\begin{Rem}
For a homogeneous space $G/H$, and $T\subset H$ a subtorus with $\rk T = \rk H = \rk G$, Proposition \ref{prop:fphomcase} tells us that the (finite) fixed point set $(G/H)^T$ of the $T$-action on $G/H$ by left multiplication is naturally given by the finite set $W(G)/W(H)$. In particular, the Weyl group $W(G)$ acts on it.

In the biquotient setting as above, in the special case $H=T$ and $\tilde T = T'$, a similar statement is true. Let $\pi:G\to G//T$ be the projection. The preimage $\pi^{-1}((G//T)^{T'})$ is equal to the set of elements $g\in G$ for which $T_{\max} g T_{\max}$ is of minimal possible dimension, or equivalently equal to $Tg$. This set clearly contains the normalizer $N_G(T_{\max})$. On the other hand, if $g$ is in this set, then  both $T_{\max}g$ and $gT_{\max}$ are equal to $T_{\max} g T_{\max}$, which implies that $g\in N_G(T_{\max})$. This implies
\[
(G//T)^{T'} = N_G(T_{\max}) //T.
\]
The normalizer $N_G(T_{\max})$ acts on this finite set, because for all $g,g'\in N_G(T_{\max})$ we have $
g\cdot Tg' = g\cdot (g'T_{\max}) = (gg')T_{\max} = Tgg'$. The subaction of $T_{\max}$ is trivial, because for $g'\in N_G(T_{\max})$ and $t\in T_{\max}$, we have $tg'\in  T_{\max} g' = Tg'$. This implies that we obtain a free and transitive action of the Weyl group $W(G)$ on $(G//T)^{T'}$.
\end{Rem}

\begin{Lemma}
In the above setting $H \cap \Delta(G)=\{(e,e)\}$.
\end{Lemma}

\begin{proof}
Take $(g,g)$ $\in H \cap \Delta G$. Then $(g,g)e=geg^{-1}=e$ and therefore $(g,g) \in H_{e}$, so $g$ equals $e$ according to the freeness of the action.
\end{proof}

\begin{Lemma}\label{lem:piequivariant}
The orbit map $\pi \colon G \longrightarrow G//H$ is $\tilde{T}$-equivariant.
\end{Lemma}

\begin{proof}
For $(t_{1},t_{2}) \in \tilde{T}$ the following is valid:
$\pi(t_{1}gt_{2}^{-1})=H(t_{1}gt_{2}^{-1})=t_{1}(Hg)t_{2}^{-1}=(t_{1},t_{2})(\pi(g))$.
\end{proof}

Now we are able to compute the isotropy representation of this action in a fixed point.\\

Let $g\in G$ be such that $Hg \in (G//H)^{\tilde{T}}$. Then, because $H$ acts freely on $G$, for each $(t_{1},t_{2}) \in \tilde{T}$ there is a unique $(s_{1},s_{2}) \in H$ such that $t_{1}gt_{2}^{-1}=s_{1}gs_{2}^{-1}$. 
\begin{Prop}
\label{isotrep} We have 
\begin{equation*}
d(t_{1},t_{2})_{Hg}(v) = d \pi_{g} d(l_{g})_{e}\Ad_{s_{2}^{-1}t_{2}}(X)
\end{equation*}
where $v \in T_{Hg}(G//H)$ and $X \in \mathfrak{g}$ satisfies $d\pi_{g}(X_g)=v$.
\end{Prop}

\begin{proof}
Since $H$ is closed under inversion, $(s_{1}^{-1},s_{2}^{-1}) \in H$. Invoking the defining equation of $(s_{1},s_{2})$ we compute using Lemma \ref{lem:piequivariant}:
\begin{align*}
d(t_{1},t_{2})_{Hg}(v) & = \frac{d}{dt}\Bigr|_{t=0} t_{1}\pi(g \cdot \expo(tX))t_{2}^{-1} \\
& = \frac{d}{dt}\Bigr|_{t=0} \pi(t_{1}(g \cdot \expo(tX))t_{2}^{-1}) \\
& = \frac{d}{dt}\Bigr|_{t=0} \pi(s_{1}^{-1}t_{1}g \cdot \expo(tX)t_{2}^{-1}s_{2}) \\
& = \frac{d}{dt}\Bigr|_{t=0} \pi(g \cdot (s_{2}^{-1}t_{2})\expo(tX)(s_{2}^{-1}t_{2})^{-1}) \\
& = d \pi_{g} d(l_{g})_{e}\Ad_{s_{2}^{-1}t_{2}}(X)
\end{align*}
\end{proof}

\begin{Lemma}
The maps $\psi:\tilde{T}\to H;\, (t_1,t_2) \mapsto (s_1,s_2)$ and  $\psi_{g} \colon \tilde{T} \longrightarrow G;\, (t_{1},t_{2}) \mapsto s_{2}^{-1}t_{2}$ are well-defined homomorphisms of Lie groups.% with finite kernel.
\end{Lemma}

\begin{Rem}
%We hide the $G//H$ component in the subscript for practical reasons.
The homomorphism $\psi_g$ depends on the choice of $g$, i.e. some representative of $Hg$.
\end{Rem}

\begin{proof}
As observed above, the freeness of the $H$-action implies that $\psi$ and $\psi_{g}$ are well-defined. Let for $(t_{1},t_{2}), (\hat{t}_{1},\hat{t}_{2}) \in \tilde{T}$ be $(s_{1},s_{2}), (\hat{s}_{1},\hat{s}_{2}) \in H$ as above. Then 
\begin{align*}
(t_{1}\hat{t}_{1},t_{2}\hat{t}_{2})g = t_{1}\hat{t}_{1}g\hat{t}_{2}^{-1}t_{2}^{-1} = t_{1}\hat{s}_{1}g\hat{s}_{2}^{-1}t_{2}^{-1} = \hat{s}_{1}t_{1}gt_{2}^{-1}\hat{s}_{2}^{-1} =\hat{s}_{1}s_{1}gs_{2}^{-1}\hat{s}_{2}^{-1},
\end{align*}
which implies that $\psi$ is a homomorphism. Further,
\begin{align*}
\psi_{g}((t_{1}\hat{t}_{1},t_{2}\hat{t}_{2}))=s_{2}^{-1}\hat{s}_{2}^{-1}t_{2}\hat{t}_{2}  =s_{2}^{-1}t_{2}\hat{s}_{2}^{-1}\hat{t}_{2}  =\psi_{g}(t_{1},t_{2})\psi_{g}(\hat{t}_{1},\hat{t}_{2}),
\end{align*}
where we used that $\tilde{T}$ and $H$ commute. It is clear that $\psi$ and $\psi_{g}$ are continuous. But every continous homomorphism of Lie groups is differentiable.
\end{proof}

For later purposes we need to determine the differential of $\psi_g$.

\begin{Lemma}
Denote by $\tau_{i}:\mathfrak{\tilde{t}}\to \mathfrak{g}$ and $\pi_i:\mathfrak{h}\to \mathfrak{g}$ the respective projections to the $i$-th factor. Furthermore we consider the maps $\alpha \colon \mathfrak{\tilde{t}} \longrightarrow \mathfrak{g}$ given by  $\alpha(X,X^{\prime})=X-X^{\prime}$ and $\beta \colon \mathfrak{h} \longrightarrow \mathfrak{g}$ given by $\beta(Y,Y^{\prime})=Y-Y^{\prime}$. Then
\begin{equation*}
d\psi_{g} = - \pi_{2} \circ \beta^{-1} \circ \alpha \circ (\Ad_{g^{-1}} \times 1) + \tau_{2}
\end{equation*}
for $g \in Hg \in (G//H)^{\tilde{T}}$.
\end{Lemma}

\begin{proof}
Writing $(s_1,s_2) = \psi(t_1,t_2)$, we have
\begin{equation*}
t_{1}gt_{2}^{-1}=s_{1}gs_{2}^{-1}.
\end{equation*}
Multiplying this equation with $g^{-1}$ from the left yields
\begin{equation*}
c_{g^{-1}}(t_{1})t_{2}^{-1}=c_{g^{-1}}(s_{1})s_{2}^{-1},
\end{equation*}
and differentiating this we obtain for $(X_1,X_2)\in \mathfrak{\tilde{t}}$ 
\begin{equation*}
\Ad_{g^{-1}}(X_{1})-X_{2}=\Ad_{g^{-1}}(\pi_{1}(d\psi(X_{1},X_{2})))-\pi_{2}(d\psi(X_{1},X_{2}))
\end{equation*}
which we can express as
\begin{equation*}
\beta((\Ad_{g^{-1}} \times 1)(d\psi(X_{1},X_{2}))=\alpha((\Ad_{g^{-1}} \times 1)(X_{1},X_{2})).
\end{equation*}
We note that $\beta$ is injective, since $\mathfrak{h}\cap \Delta \mathfrak{g}=\ker(\beta)=0$, its image contains $\mathfrak{t}_{\max}$ and $\alpha$ has image contained in $\mathfrak{t}_{\max}$. Therefore we have
\begin{equation*}
d\psi=(\Ad_{g^{-1}} \times 1)^{-1} \circ \beta^{-1} \circ \alpha \circ (\Ad_{g^{-1}} \times 1)
\end{equation*}
Now we can use this to differentiate the homomorphism $\psi_g$, which was given by $\psi_g(t_1,t_2) = s_2^{-1}t_2$: it is 
\begin{align*}
d\psi_{g} = - \pi_{2} \circ \beta^{-1} \circ \alpha \circ (\Ad_{g^{-1}} \times 1) + \tau_{2} \\
% = \tau_{1} \circ (\Ad_{g^{-1}} \times 1) + \tau_{2}
\end{align*}
which completes our proof.
\end{proof}

\begin{Cor}\label{Corollary weights}
If $T'$ lies in the special torus $\{(t_1,t_2)\in G\times G\mid(t_2,t_2) \in T\}$, this differential computes as 
\begin{equation*}
d\psi_{g} 
= \tau_{1} \circ (\Ad_{g^{-1}} \times 1) + \tau_{2}.
\end{equation*}
\end{Cor}

\begin{proof}
In this case we have $ - \pi_{2} \circ \beta^{-1} \circ \alpha=\tau_{1}$.
\end{proof}

\begin{Cor}\label{weightcorollary}
If we fix an auxiliary biinvariant Riemannian metric on $G$ and denote by $\hat{\Delta}_{g}$ the set of weights of the restriction of the adjoint representation of $G$ on $\mathfrak{g}$ to the subspace $d(l_{g^{-1}})_{e}(\ker d\pi_{g})^{\perp}$ and the subtorus ${\mathrm{Im}}(\psi_{g})$, the set of weights of the isotropy representation in the fixed point $Hg$ is $\Delta_{g} := \{d(\psi_{g})^{*}\lambda | \lambda \in \hat{\Delta}_{g}\}$. 
\end{Cor}

\begin{proof}
In Proposition \ref{isotrep} we proved the commutativity of the following diagram:
\begin{center}
\begin{tikzcd}
\mathfrak{g} \arrow{r}{d(l_{g})_{e}} \arrow{d}[swap]{\Ad_{\psi_{g}(t_{1},t_{2})}} &T_{g}G \arrow{r}{d\pi_{g}}  &T_{Hg}(G//H) \arrow{d}{d(t_{1},t_{2})_{Hg}} \\
\mathfrak{g} \arrow{r}{d(l_{g})_{e}} &T_{g}G \arrow{r}{d\pi_{g}} &T_{Hg}(G//H).
\end{tikzcd}
\end{center}
In order to get isomorphic representations we fix a biinvariant Riemannian metric on $G$, restrict to appropriate subspaces and finally achieve the following diagram:

\begin{center}
\begin{tikzcd}
(d(l_{g^{-1}})_{g}(\ker d\pi_{g}))^{\perp} \arrow{r}{d(l_{g})_{e}} \arrow{d}[swap]{\Ad_{\psi_{g}(t_{1},t_{2})}} &(\ker d\pi_{g})^{\perp} \arrow{r}{d\pi_{g}} &T_{Hg}(G//H) \arrow{d}{d(t_{1},t_{2})_{Hg}} \\
(d(l_{g^{-1}})_{g}(\ker d\pi_{g}))^{\perp} \arrow{r}{d(l_{g})_{e}} &(\ker d\pi_{g})^{\perp} \arrow{r}{d\pi_{g}} &T_{Hg}(G//H).
\end{tikzcd}
\end{center}
 
The weights of the above twisted adjoint representation are then the twisted weights $\{d(\psi_{g})^{*}\lambda | \lambda \in \hat{\Delta}_{Hg}\}$.
\end{proof}

\begin{Rem}\label{maximaltorusimage}
The most convenient situation occurs, when $\tilde{T}$ lies in $T_{\max} \times T_{\max}$ and for each fixed point $Hg \in (G//H)^{\tilde{T}}$ there exists a representative $g \in N_{G}(T_{\max})$. Then ${\mathrm{Im}}(\psi_{g})$ lies in $T_{\max}$ and the weights are pulled back roots associated to the maximal torus $T_{\max}$.
\end{Rem}

\begin{Def}\label{weightremark}
The weights are only well-defined up to sign. If we fix an orientation on $G//H$, we denote by $\Delta_{g}^{+}$ the set of weights $\Delta_{g}$ with fixed signs, such that the oriented weight space decomposition 
\[
T_{Hg}(G//H) \cong \bigoplus\limits_{\alpha \in \Delta^{+}_{g}} T_{Hg}(G//H)_{\alpha},
\]
where $T_{Hg}(G//H)_{\alpha}$ is the weight space corresponding to the weight $\alpha$, induces the set orientation on $T_{Hg}(G//H)$.
\end{Def}

\section{Signature}
Just as in the homogeneous case we can now invoke Hirzebruch's signature formula to prove a result on the signature of biquotients.

\begin{Th}\label{signature}
Suppose that the fixed point set of $\tilde{T} \curvearrowright G//H$ consists of isolated points and fix $(X,Y) \in \tilde{\mathfrak{t}}$ generating a subcircle with the same fixed points. Then:
\begin{equation*}
\sigma(G//H)= \pm \sum\limits_{Hg \in (G//H)^{\tilde{T}}}(-1)^{\#\{\alpha \in \Delta_{g}^{+} \mid \alpha(X,Y)<0\}}
\end{equation*}

\end{Th}

\begin{proof}
Since $G//H$ is compact, the fixed point set is finite.
Fixing an orientation on $G//H$, while having Corollary \ref{weightcorollary} and Definition \ref{weightremark} in mind, carries us directly to the situation of Theorem \ref{hirzebruchformula}. We can apply Hirzebruch's Theorem \ref{hirzebruchformula} for oriented $S^{1}$-manifolds which implies the announced formula.
\end{proof}

\begin{Rem}\label{orientationremark}
By \cite[Corollary 3.4.\ and Property\ 1.7.]{singhof} $G//H$ is orientable whenever $G$ and $H$ are connected. In that case, we can orient $G//H$ as follows. By introducing a bi-invariant auxiliary Riemannian metric on $G$ we can make the following identifications:
\begin{align*}
T_{Hg}G//H &\cong (\ker d\pi_{g})^{\perp} \\
           &\cong (d(l_{g^{-1}})_{g}(\ker d\pi_{g}))^{\perp} \\
           &\cong (d(l_{g^{-1}})_{g}(T_{g}H \cdot g)^{\perp} \\
           &\cong \{\Ad_{g^{-1}}X-Y \mid (X,Y) \in T_{e}H\}^{\perp},
\end{align*}
which gives us a splitting
\begin{equation*}
\mathfrak{g} \cong T_{Hg}G//H \oplus \{\Ad_{g^{-1}}X-Y \mid (X,Y) \in T_{e}H\}.
\end{equation*}
Therefore fixing orientations of $G$ and $H$ we get an orientation of each orbit $H \cdot g$ and  an induced orientation of its normal space $\nu(H \cdot g)$, which is by the previous considerations isomorphic to $T_{Hg}G//H$. Note that the orientation of the orbit
\begin{equation*}
(d(l_{g^{-1}})_{g}(T_{g}H\cdot g) \cong \{\Ad_{g^{-1}}X-Y \mid (X,Y) \in T_{e}H\}
\end{equation*}
is independent of the choice of the representative of the orbit because $H$ is connected. Hence we can determine an orientation of the biquotient $G//H$, by choosing sets of positive roots of $G$ and $H$ and orientations on their maximal tori.
\end{Rem}

\begin{Rem}\label{rem:signaturevanishes}
Let us describe two situations in which the signature of a biquotient vanishes automatically: For $\rk(H)<\rk(G)$ the signature behaves analogously to the homogeneous case and $\sigma(G//H)=0$ because by \cite[Proposition\ 6.7.]{singhof} all Pontryjagin numbers of $G//H$ vanish and therefore the signature vanishes by Hirzebruch's signature theorem \cite[Theorem\ 8.2.2]{HirzTop}.

Consider a biquotient of the form $G//T$, where $G$ is a compact simple Lie group and $T\subset G\times G$ is a torus with $\rk T = \rk G$. Such biquotients were classified by Eschenburg in \cite[Chapters 6,7,8]{Esch2} (up to a certain notion of equivalence). Moreover, it follows from the results in Chapter 9 of the same reference that there always exists a nonabelian extension $T\subset H \subset G\times G$ with $\rk H = \rk G$ (in fact, there the maximal such extensions are classified). In particular, we obtain a fibration
\begin{center}
$H/T \longrightarrow G//T \longrightarrow G//H$,
\end{center} 
cf.\ \cite[Section 2.1]{GKZ2}, from which we obtain $\sigma(G//T) = \sigma(H/T)\sigma(G//H)$ by \cite{index}. But the signature of the generalized flag manifold $H/T$ vanishes by \cite[Proposition\ 2.4]{Hirzebruch}, which implies that $\sigma(G//T) = 0$.
\end{Rem}

\subsection{An Example}\label{subsec:ex}
Let us apply Theorem \ref{signature} to an example.
Take $G=\SU(6)$ and let $H=\Delta^{3}(\SU(2)) \times \SU(5)\subset G\times G$, where $\Delta^{3}(\SU(2))=\left\{ \left.\begin{pmatrix}
A &0  &0 \\
0 &A &0 \\
0 &0 &A
\end{pmatrix}    \right| A \in \SU(2)\right\}$ is the blockwise embedding and $\SU(5)$ is embedded in the upper left corner. Let $T\subset H$ be the maximal torus given by diagonal matrices in both components. We will compute the signature of the biquotient $G//H$, in order to illustrate our formula. This will not be a new result; as $G//H=\Delta^{3}(\SU(2)) \backslash \SU(6) /\SU(5) \cong \Delta^{3}(\SU(2)) \backslash S^{11} \cong \mathbb{H}P^{2}$, the signature is well-known to be $\pm 1$.

The first step is to find a subtorus of $G\times G$ which commutes with $H$ and acts with finite fixed point set on $G//H$, and determine the weights of the isotropy representation in each fixed point.
Such a torus is for example given by $\tilde{T}=\{\diag(\lambda,\lambda,\lambda^{-1},\lambda^{-1},1,1) | \lambda \in S^{1}\}\times \{1\}$. We note that $\tilde{T}$ is contained in the flipped torus $T'=\{(t_1,t_2)\mid (t_2,t_1)\in T\}$. It is easily seen that the action of $\tilde{T}$ on $G//H \cong \mathbb{H}P^{2}$ is given by $\lambda \cdot [q_{1}:q_{2}:q_{3}]=[\lambda q_{1}: \lambda^{-1}q_{2}:q_{3}]$ because the diffeomorphism $\SU(6)/\SU(5) \cong S^{11}$ is just projection on the last column. Hence our fixed point set is $(G//H)^{\tilde{T}}=\{[1:0:0],[0:1:0],[0:0:1]\}=$\\

$\left\{H \cdot \begin{pmatrix}
0 &0 &0 &0 &0 &1\\
-1 &0 &0 &0 &0 &0\\
0 &1 &0 &0 &0 &0\\
0 &0 &1 &0 &0 &0\\
0 &0 &0 &1 &0 &0\\
0 &0 &0 &0 &1 &0
\end{pmatrix}, H \cdot \begin{pmatrix}
1 &0 &0 &0 &0 &0\\
0 &-1 &0 &0 &0 &0\\
0 &0 &0 &0 &0 &1\\
0 &0 &1 &0 &0 &0\\
0 &0 &0 &1 &0 &0\\
0 &0 &0 &0 &1 &0
\end{pmatrix}, H \cdot \begin{pmatrix}
1 &0 &0 &0 &0 &0\\
0 &-1 &0 &0 &0 &0\\
0 &0 &1 &0 &0 &0\\
0 &0 &0 &1 &0 &0\\
0 &0 &0 &0 &0 &1\\
0 &0 &0 &0 &1 &0
\end{pmatrix}\right\}$\\

We define $g_{1},g_{2},g_{3}$ as the above representatives of the fixed points. Note that we are in the situation of Remark \ref{maximaltorusimage}.

Throughout this example, we denote by $V_{jk}\subset \su(6)$, where $j,k=1,\ldots,6$, $j\neq k$, the span of $E_{ij} - E_{ji}$ and $i(E_{ij} + E_{ji})$. This is the root space of the adjoint representation of the standard maximal torus on $\su(6)$ of the root $\pm (e_i-e_j)$. By choosing the set of positive roots $\{e_{i}-e_{j} \mid i<j\}$ we induce an orientiation on $V_{ij}$, with respect to which the above fixed basis is positively oriented. We thus obtain an orientation on $\su(6) = {\mathfrak{t}}_{\max} \oplus \bigoplus_{i<j} V_{ij}$ by declaring the basis $\{i(E_{11} - E_{66}),\ldots, i(E_{55}-E_{66})\}$ of ${\mathfrak{t}}_{\max}$ to be positively oriented. Analogously we obtain an orientation on $\su(2)$, $\su(5)$, and then also on
\[
\su(2)\times \su(5) = {\mathfrak{t}} \oplus (V_{12}\times 0) \oplus \bigoplus_{1\leq i<j\leq 5} (0\times V_{ij}),
\]
via the positively oriented basis $(i(E_{11}-E_{22}),0),(0,i(E_{11}-E_{55}),\ldots,(0,i(E_{44}-E_{55}))\}$. These orientations on $G$ and $H$ induce an orientation on $G//H$, cf.\ Remark \ref{orientationremark}.

Using the Frobenius inner product or equivalently the Killing form on $\SU(6)$ we can determine the complements $\ker(d\pi)_{g_{i}}^{\perp} \cong T_{Hg_{i}}G//H$. 
We obtain 
\[
d(l_{g_{1}})_{e}^{-1}\ker(d\pi)_{g_{1}}^{\perp} =\Big\{\begin{pmatrix}
0 &0 &0 &0 &0 &0\\
0 &0 &0 &0 &0 &\ast\\
0 &0 &0 &0 &0 &\ast\\
0 &0 &0 &0 &0 &\ast\\
0 &0 &0 &0 &0 &\ast\\
0 &\ast &\ast &\ast &\ast &0
\end{pmatrix} \Big\} = V_{26} \oplus V_{36} \oplus V_{46} \oplus V_{56}\subset \su(6),
\]
\[
d(l_{g_{2}})_{e}^{-1}\ker(d\pi)_{g_{2}}^{\perp} =\Big\{\begin{pmatrix}
0 &0 &0 &0 &0 &\ast\\
0 &0 &0 &0 &0 &\ast\\
0 &0 &0 &0 &0 &0\\
0 &0 &0 &0 &0 &\ast\\
0 &0 &0 &0 &0 &\ast\\
\ast &\ast &0 &\ast &\ast &0
\end{pmatrix} \Big\} = V_{16} \oplus V_{26} \oplus V_{46} \oplus V_{56}\subset \su(6),
\]
\[
d(l_{g_{3}})_{e}^{-1}\ker(d\pi)_{g_{3}}^{\perp} =\Big\{\begin{pmatrix}
0 &0 &0 &0 &0 &\ast\\
0 &0 &0 &0 &0 &\ast\\
0 &0 &0 &0 &0 &\ast\\
0 &0 &0 &0 &0 &\ast\\
0 &0 &0 &0 &0 &0\\
\ast &\ast &\ast &\ast &0 &0
\end{pmatrix} \Big\} = V_{16} \oplus V_{26} \oplus V_{36} \oplus V_{46}\subset \su(6).
\]
By Corollary \ref{weightcorollary}, the weights of the $\tilde T$-isotropy representation in the three fixed points are 
\begin{itemize}
\item $\Delta_{g_{1}}=\{\pm d\psi_{g_{1}}^{*}(\frac{1}{i}(e_{2}-e_{6})),\pm d\psi_{g_{1}}^{*}(\frac{1}{i}(e_{3}-e_{6})),\pm d\psi_{g_{1}}^{*}(\frac{1}{i}(e_{4}-e_{6})),\pm d\psi_{g_{1}}^{*}(\frac{1}{i}(e_{5}-e_{6}))\}$
\item $\Delta_{g_{2}}=\{\pm d\psi_{g_{2}}^{*}(\frac{1}{i}(e_{1}-e_{6})),\pm d\psi_{g_{2}}^{*}(\frac{1}{i}(e_{2}-e_{6})),\pm d\psi_{g_{2}}^{*}(\frac{1}{i}(e_{4}-e_{6})),\pm d\psi_{g_{2}}^{*}(\frac{1}{i}(e_{5}-e_{6}))\}$
\item $\Delta_{g_{3}}=\{\pm d\psi_{g_{3}}^{*}(\frac{1}{i}(e_{1}-e_{6})),\pm d\psi_{g_{3}}^{*}(\frac{1}{i}(e_{2}-e_{6})),\pm d\psi_{g_{3}}^{*}(\frac{1}{i}(e_{3}-e_{6})),\pm d\psi_{g_{3}}^{*}(\frac{1}{i}(e_{4}-e_{6}))\}$.
\end{itemize}
where we now denote by $\frac{1}{i}(e_i-e_j)$ the restrictions of the realifications of the usual  roots to the tori ${\mathrm{Im}}(\psi_{g_k})$. We now have to choose appropriate signs of these weights, i.e., define compatible sets of weights $\Delta_{g_k}^+$ as in Definition \ref{weightremark}.

For every $k$, the subspace $\{\Ad_{g_{i}^{-1}}X-Y \mid (X,Y) \in T_{e}H\}\subset \su(6)$ is the sum of the Lie algebra of the maximal torus of $\su(6)$ and certain root spaces, and hence oriented by our conventions above. Using the bases above, and taking into account the embeddings of $\su(2)$ and $\su(5)$ into $\su(6)$, in order to define $\Delta_{g_k}^+$ we have to determine if the  natural maps
\begin{equation}\label{eq:orbitmapexample}
\su(2)\times \su(5)\longrightarrow   \{\Ad_{g_{k}^{-1}}X-Y \mid (X,Y) \in T_{e}H\} \subset \su(6).
\end{equation}
are orientation-preserving. The images of the embedded basis of $\su(2)$ are
\begin{align*}
&\Ad_{g_{1}^{-1}}(\diag(i,-i))&&= \diag(-i,i,-i,i,-i,i)\\
&\Ad_{g_{1}^{-1}}(E_{12}-E_{21})&&= (E_{23}-E_{32}) + (E_{45}-E_{54}) + (E_{16}-E_{61})\\
&\Ad_{g_{1}^{-1}}(i(E_{12}+E_{21}))&&= i(E_{23}+E_{32}) + i(E_{45}+E_{54}) - i(E_{16}+E_{61})\\
&\Ad_{g_{2}^{-1}}(\diag(i,-i))&&=\diag(i,-i,-i,i,-i,i) \\
&\Ad_{g_{2}^{-1}}(E_{12}-E_{21})&&=-(E_{12}-E_{21}) + (E_{45}-E_{54}) - (E_{36}-E_{63}) \\
&\Ad_{g_{2}^{-1}}(i(E_{12}+E_{21}))&&= -i(E_{12}+E_{21}) + i(E_{45}+E_{54}) + i(E_{36}+E_{63}) \\ 
&\Ad_{g_{3}^{-1}}(\diag(i,-i))&&=  \diag(i,-i,i,-i,-i,i)\\
&\Ad_{g_{3}^{-1}}(E_{12}-E_{21})&&= -(E_{12}-E_{21}) + (E_{34}-E_{43}) - (E_{56}-E_{65}) \\
&\Ad_{g_{3}^{-1}}(i(E_{12}+E_{21}))&&=-i(E_{12}+E_{21}) + i(E_{34}+E_{43}) + i(E_{56}+E_{65}) .
\end{align*}
Moreove, everything from the $\su(5)$ factor is mapped to its negative. From this, one computes the map \eqref{eq:orbitmapexample}:
\begin{itemize}
\item For $g_1$, it is the direct sum of an orientation-reversing map $\mathfrak t\to {\mathfrak t}_{\max}$ and an orientation-preserving map $(V_{12}\times 0) \oplus \bigoplus_{1\leq i<j\leq 5} (0\times V_{ij})\to V_{16} \oplus \bigoplus_{1\leq i<j\leq 5} V_{ij}$.
\item For $g_2$, it is the direct sum of an orientation-reversing map $\mathfrak t\to {\mathfrak t}_{\max}$ and an orientation-reversing map $(V_{12}\times 0) \oplus \bigoplus_{1\leq i<j\leq 5} (0\times V_{ij})\to V_{36} \oplus \bigoplus_{1\leq i<j\leq 5} V_{ij}$.
\item For $g_3$, it is the direct sum of an orientation-reversing map $\mathfrak t\to {\mathfrak t}_{\max}$ and an orientation-reversing map $(V_{12}\times 0) \oplus \bigoplus_{1\leq i<j\leq 5} (0\times V_{ij})\to V_{56} \oplus \bigoplus_{1\leq i<j\leq 5} V_{ij}$.
\end{itemize}
Thus, for $g_2$ and $g_3$ the original orientation given by that of the $V_{ij}$ is the correct one on $T_{Hg_{k}}G//H$, while for $g_1$ we have to take the opposite one. We can therefore fix the following sets of weights of $\d(l_{g_{i}^{-1}})_{g_{i}}(T_{g_{i}}Hg_{i})^{\perp}$ for each fixed point $g_{i}$ inducing the fixed orientation on $G//H$:

\begin{itemize}
\item $\Delta_{g_{1}}^{+}=\{-d\psi_{g_{1}}^{*}(\frac{1}{i}(e_{2}-e_{6})),d\psi_{g_{1}}^{*}(\frac{1}{i}(e_{3}-e_{6})),d\psi_{g_{1}}^{*}(\frac{1}{i}(e_{4}-e_{6})),d\psi_{g_{1}}^{*}(\frac{1}{i}(e_{5}-e_{6}))\}$
\item $\Delta_{g_{2}}^{+}=\{d\psi_{g_{2}}^{*}(\frac{1}{i}(e_{1}-e_{6})),d\psi_{g_{2}}^{*}(\frac{1}{i}(e_{2}-e_{6})),d\psi_{g_{2}}^{*}(\frac{1}{i}(e_{4}-e_{6})),d\psi_{g_{2}}^{*}(\frac{1}{i}(e_{5}-e_{6}))\}$
\item $\Delta_{g_{3}}^{+}=\{d\psi_{g_{3}}^{*}(\frac{1}{i}(e_{1}-e_{6})),d\psi_{g_{3}}^{*}(\frac{1}{i}(e_{2}-e_{6})),d\psi_{g_{3}}^{*}(\frac{1}{i}(e_{3}-e_{6})),d\psi_{g_{3}}^{*}(\frac{1}{i}(e_{4}-e_{6}))\}$.
\end{itemize}

Furthermore, because by our choices $\tilde{T}$ lies inside the flipped torus $T'$, Corollary \ref{Corollary weights} applies, and 
\[
d\psi_{g_{k}}(X,Y)=\Ad_{g_{k}^{-1}}(X)+Y.
\]

If we now choose $(iX,0)\in i \cdot \mathbb{R} \times 0 \cong Lie(S^{1} \times 1)$, $X>0$ generating $\tilde{T}$ , we compute invoking Corollary \ref{Corollary weights}

\begin{enumerate}
\item[]
\begin{equation*}
g_{1}^{-1}(iX)g_{1}= \begin{pmatrix}
iX &0 &0 &0 &0 &0\\
0 &-iX &0 &0 &0 &0\\
0 &0 &-iX &0 &0 &0\\
0 &0 &0 &0 &0 &0\\
0 &0 &0 &0 &0 &0\\
0 &0 &0 &0 &0 &iX
\end{pmatrix} \Rightarrow 
	\begin{aligned}
	-(\frac{1}{i}(e_{2}-e_{6}))(\Ad_{g_{1}^{-1}}(iX))&=2X>0\\
	\frac{1}{i}(e_{3}-e_{6})(\Ad_{g_{1}^{-1}}(iX))&=-2X<0\\
	\frac{1}{i}(e_{4}-e_{6})(\Ad_{g_{1}^{-1}}(iX))&=-X<0\\
	\frac{1}{i}(e_{5}-e_{6})(\Ad_{g_{1}^{-1}}(iX))&=-X<0
\end{aligned}
\end{equation*}

\item[]
\begin{equation*}
g_{2}^{-1}Xg_{2}=\begin{pmatrix}
iX &0 &0 &0 &0 &0\\
0 &iX &0 &0 &0 &0\\
0 &0 &-iX &0 &0 &0\\
0 &0 &0 &0 &0 &0\\
0 &0 &0 &0 &0 &0\\
0 &0 &0 &0 &0 &-iX
\end{pmatrix} \Rightarrow 
	\begin{aligned}
	\frac{1}{i}(e_{1}-e_{6})(\Ad_{g_{2}^{-1}}(iX))&=2X>0\\
	\frac{1}{i}(e_{2}-e_{6})(\Ad_{g_{2}^{-1}}(iX))&=2X>0\\
	\frac{1}{i}(e_{4}-e_{6})(\Ad_{g_{2}^{-1}}(iX))&=X>0\\
	\frac{1}{i}(e_{5}-e_{6})(\Ad_{g_{2}^{-1}}(iX))&=X>0
\end{aligned}
\end{equation*}
\item[]
\begin{equation*}
g_{3}^{-1}Xg_{3}=\begin{pmatrix}
iX &0 &0 &0 &0 &0\\
0 &iX &0 &0 &0 &0\\
0 &0 &-iX &0 &0 &0\\
0 &0 &0 &-iX &0 &0\\
0 &0 &0 &0 &0 &0\\
0 &0 &0 &0 &0 &0
\end{pmatrix} \Rightarrow 
	\begin{aligned}
	\frac{1}{i}(e_{1}-e_{6})(\Ad_{g_{3}^{-1}}(iX))&=X>0\\
	\frac{1}{i}(e_{2}-e_{6})(\Ad_{g_{3}^{-1}}(iX))&=X>0\\
	\frac{1}{i}(e_{3}-e_{6})(\Ad_{g_{3}^{-1}}(iX))&=-X<0\\
	\frac{1}{i}(e_{4}-e_{6})(\Ad_{g_{3}^{-1}}(iX))&=-X<0
\end{aligned}
\end{equation*}
\end{enumerate}

We can now apply Theorem \ref{signature} and obtain:
\begin{equation*}
\sigma(G//H)=\pm ((-1)^{3}+(-1)^{0}+(-1)^{2})=\pm 1.
\end{equation*}


\begin{thebibliography}{9}

%\bibitem{Audin}	 	
%	M. Audin:
%	\textit{Torus actions on symplectic manifolds},
%	Second revised edition. Progress in Mathematics, 93. Birkhaeuser Verlag, %Basel, 2004. viii+325 pp. 
%	ISBN: 3-7643-2176-8 

%\bibitem{Panorama}
%	M. Berger:
%	\textit{A panoramic view of Riemannian geometry},
%	Springer-Verlag, Berlin, 2003. xxiv+824 pp. ISBN: 3-540-65317-1

\bibitem{App1}
	J. Bliss, R. Moody, A. Pianzola:
	\textit{Appendix to: "Elliptic genera, involutions, and homogeneous spin manifolds" by F. 		
	Hirzebruch and P. Slodowy},
	Geom. Dedicata 35 (1990), no. 1-3, 345-351. 
	
%\bibitem{borel}
%	A. Borel:
%	\textit{Seminar on transformation groups},
%	With contributions by G. Bredon, E. E. Floyd, D. Montgomery, R. Palais. %Annals of Mathematics Studies, No. 46 Princeton University Press, %Princeton, N.J. 1960 vii+245 pp.	
	
%\bibitem{botttu}
%	R. Bott, L. Tu:
%	\textit{Differential forms in algebraic topology},
%	Graduate Texts in Mathematics, 82. Springer-Verlag, New York-Berlin, %1982. xiv+331 pp. ISBN: 
%	0-387-90613-4	
	
\bibitem{boyer}
	C. Boyer, K. Galicki, B. Mann:
	\textit{The geometry and topology of 3-Sasakian manifolds},
	J. Reine Angew. Math. 455 (1994), 183-220. 	
	
%\bibitem{index}
%	S.S. Chern, F. Hirzebruch, J.-P. Serre:
%	\textit{On the index of a fibered manifold},
%	Proc. Amer. Math. Soc. 8 (1957), 587-596. 

\bibitem{Esch1}
	J.-H. Eschenburg: 
	\textit{Cohomology of biquotients},
	Manuscripta Math. 75 (1992), no. 2, 151-166.

\bibitem{Esch2}
	J.-H. Eschenburg:
	\textit{Freie isometrische Aktionen auf kompakten Liegruppen mit positiv gekr\"ummten Orbitr\"aumen}
	Schriftenreihe des Mathematischen Instituts der Universit\"at M\"unster, 2. Serie, 32. Universit\"at    	
	M\"unster, Mathematisches Institut, M\"unster, 1984.
	
%\bibitem{Esch3}
%	J.-H. Eschenburg:
%	\textit{Inhomogeneous spaces of positive curvature},
%	Differential Geom. Appl. 2 (1992), no. 2, 123-132.

%\bibitem{Esch4}
%	J.-H. Eschenburg:
%	\textit{New examples of manifolds with strictly positive curvature},
%	Invent. Math. 66 (1982), no. 3, 469-480. 	
	
\bibitem{GKZ2}	
	O. Goertsches, P. Konstantis, L. Zoller:
	\textit{Symplectic and K\"ahler structures on biquotients},
	preprint, arXiv:1812.09689, to appear in J.\ Symplectic Geom.
	
\bibitem{GKZ1}
	O. Goertsches, P. Konstantis, L. Zoller:
	\textit{GKM theory and Hamiltonian non-K\"ahler actions in dimension 6},
	 Adv. Math. 368 (2020), 107141, 17 pp.
		


%\bibitem{GKM}
%	M. Goresky, R. Kottwitz, R. McPherson:
%	\textit{Equivariant cohomology, Koszul duality, and the localization %theorem},
%	Invent. Math. 131 (1998), no. 1, 25-83.      	
	
%\bibitem{gromollmayer}
%	D. Gromoll, W. Mayer:
%	\textit{An exotic sphere with nonnegative sectional curvature},
%	Ann. of Math. (2) 100 (1974), 401-406. 
		
		
\bibitem{GHZ}
	V. Guillemin, T. Holm, C. Zara:
  	\textit{A GKM description of the equivariant cohomology ring of a homogeneous space},
  	J. Algebraic Combin. 23 (2006), no. 1, 21-41. 

%\bibitem{GS}
%	V. Guillemin, S. Sternberg:
%	\textit{Supersymmetry and equivariant de Rham theory},
%	With an appendix containing two reprints by Henri Cartan. Mathematics %Past and Present. Springer-Verlag, Berlin, 1999. xxiv+228 pp. ISBN: %3-540-64797-X.

\bibitem{HirzTop}
	F. Hirzebruch
	\textit{Topological Methods in Algebraic Geometry},
	Translated from the German and Appendix One by R. L. E. Schwarzenberger. With a preface to the 	third English edition by the author and Schwarzenberger. Appendix Two by A. Borel. Reprint of the 1978 edition. Classics in Mathematics. Springer-Verlag, Berlin, 1995. 


\bibitem{hirzbruch}	
	F. Hirzebruch, T. Berger, R. Jung:
	\textit{Manifolds and modular forms},
	 Aspects of Mathematics, E20. Friedr. Vieweg \& Sohn, Braunschweig, 1992. 


\bibitem{Hirzebruch}
  	F. Hirzebruch, P. Slodowy:
  	\textit{Elliptic genera, Involutions and Homogeneous Spin Manifolds},
  	Geom. Dedicata 35 (1990), no. 1-3, 309-343. 

\bibitem{rathomot}
	V. Kapovitch:
	\textit{A note on rational homotopy of biquotients},
	preprint.	

%\bibitem{kapovitchziller}
%     V. Kapovitch, W. Ziller:
%     \textit{Biquotients with singly generated rational cohomology},
%     Geom. Dedicata 104 (2004), 149-160. 
     
%\bibitem{kotschick}    
%    D. Kotschick, S. Terzic:
%    \textit{Geometric formality of homogeneous spaces and of biquotients},
%    Pacific J. Math. 249 (2011), no. 1, 157-176.      

%\bibitem{LEE}
%	J. Lee:
%	\textit{Introduction to Smooth Manifolds},
%	Second edition. Graduate Texts in Mathematics, 218. Springer, New York, %2013. xvi+708 pp. ISBN: 		978-1-4419-9981-8.     

\bibitem{index}
J. A. Schafer, \textit{The signature of fiber bundles}, Proc. Amer. Math. Soc. 33 (1972), 548--550.

\bibitem{singhof}
	W. Singhof: 
	\textit{On the topology of double coset manifolds},
	Math. Ann. 297 (1993), no. 1, 133-146. 


\bibitem{App2}
	P. Slodowy:
	\textit{On the signature of homogeneous spaces},
	Geom. Dedicata 43 (1992), no. 1, 109-120. 
		
%\bibitem{ziller}
%	W. Ziller:
%	\textit{Examples of Riemannian manifolds with non-negative sectional %curvature},
%	Surveys in differential geometry. Vol. XI, 63-102, Surv. Differ. Geom., %11, Int. Press, 		
%	Somerville, MA, 2007.

\end{thebibliography}
\end{document}